\documentclass[11pt]{amsart}

\usepackage{amsmath}
\usepackage{amssymb}
\usepackage{epsfig}
\usepackage{amsfonts}
\usepackage{ytableau}
\ytableausetup{boxsize=1.2em}

\newcommand{\excise}[1]{}%$\star$\textsc{#1}$\star$}

\newtheorem{thm}{Theorem}%[section]
\newtheorem{lemma}[thm]{Lemma}

\newtheorem{conv}[thm]{Convention}

\newtheorem{rem}[thm]{Remark}
\newtheorem{Warn}[thm]{Caution}

    {\end{Example}}
    {\end{Remark}}

                     {\end{minipage}\end{center}\smallskip}

\newenvironment{Abs}{\smallskip\begin{center}\begin{minipage}{13.8cm}}%
                     {\end{minipage}\end{center}\smallskip}

\def\la{\lambda}
\def\ga{\gamma}

\def\al{\alpha}
\def\be{\beta}

\def\<{\langle}
\def\>{\rangle}

\def\0{{\mathbf 0}}

\def\.{\hskip.06cm}
\def\ts{\hskip.03cm}

\def\nin{\noindent}

\begin{document}
\title[Strict unimodality]{Strict unimodality of $q$-binomial coefficients}

\author[Igor~Pak]{ \ Igor~Pak$^\star$}

\author[Greta~Panova]{ \ Greta~Panova$^\star$}

\date{\today}

\thanks{\thinspace ${\hspace{-.45ex}}^\star$Department of Mathematics, UCLA, Los Angeles, CA 90095, USA; \ts
\texttt{\{pak,panova\}@math.ucla.edu}}

\maketitle

\begin{Abs}{\footnotesize {\sc Abstract.} \ts
We prove strict unimodality of the $q$-binomial coefficients~$\ts \binom{n}{k}_q$
as polynomials in~$q$.  The proof is based on the combinatorics of certain
Young tableaux and the semigroup property of Kronecker coefficients
of $S_n$~representations. }
\end{Abs}

%\smallskip
%
%\begin{Abs} {\footnotesize {\sc R\'{e}sum\'{e}.}
%Nous prouvons l'unimodalit\'{e} stricte des coefficients $q$-binomiaux ~ $ \ts \binom{n}{k}_q $
%comme des polyn\^{o}mes en ~ $ q $. La preuve est bas\'{e}e sur la combinatoire de certains
%Tableaux de Young et la propri\'{e}t\'{e} semi-groupe des coefficients Kronecker
%de $ S_n $ ~ repr\'{e}sentations.
%
% }
%\end{Abs}

\bigskip

\section*{Introduction}

\noindent
A sequence $(a_1,a_2,\ldots,a_n)$ is called \emph{unimodal}, if for some~$k$ we have
$$
a_1 \, \le \, a_2 \, \le \,  \ldots \, \le \,  a_k \, \ge \,  a_{k+1}
\, \ge \, \ldots \, \ge \, a_n\ts.
$$
The \emph{$q$-binomial}~(Gaussian)~\emph{coefficients} are defined as:
$$
\binom{m+\ell}{m}_q
\, = \ \. \frac{(q^{m+1}-1)\. \cdots\. (q^{m+\ell}-1)}{(q-1)\.\cdots\. (q^{\ell}-1)}
\ \. = \, \, \sum_{n=0}^{\ell\ts m} \, \. p_n(\ell,m) \. q^n\ts.
$$
\emph{Sylvester's theorem} establishes unimodality of the sequence
$$
p_0(\ell,m)\ts, \, p_1(\ell,m)\ts, \, \ldots \,, \, p_{\ell\ts m}(\ell,m)\ts.
$$
This celebrated result was first conjectured by Cayley in~1856, and proved
by Sylvester using Invariant Theory, in a pioneer~1878 paper~\cite{Syl}.
In the past decades, a number of new proofs and generalizations were
discovered both by algebraic and combinatorial tools, see Section~\ref{s:fin}.
In the previous paper~\cite{PP}, we found
a new proof of Sylvester's theorem using combinatorics of Kronecker
and Littlewood--Richardson coefficients.  Here we use the
recently established \emph{semigroup property} of Kronecker
coefficients to prove \emph{strict unimodality} of
$q$-binomial coefficients:

\smallskip

\begin{thm}\label{t:strict}
For all $\.\ell,m\ge 8\ts$, we have the following strict inequalities:
$$(\circ) \qquad
p_{1}(\ell,m)\. < \. \ldots \. < \. p_{\lfloor \ell\ts m/2\rfloor}(\ell,m)
\. = \. p_{\lceil \ell\ts m/2\rceil}(\ell,m) \. > \. \ldots \. > \. p_{\ell\ts m-1}(\ell,m)\ts.
$$
\end{thm}

\smallskip

\noindent
These and the remaining cases are covered in Theorem~\ref{t:rest}.
Note that neither combinatorial nor algebraic tools imply~$(\circ)$ directly,
as Sylvester's theorem is notoriously hard to prove and extend.
In our previous paper~\cite{PP}, we proved strict unimodality of the diagonal
coefficients $\.\binom{2\ts m}{m}_q$  by combining technical algebraic
tools from~\cite{PPV} and Almkvist's analytic unimodality results.

% Theorem~\ref{t:rest} below covers remaining cases of strict unimodality.

\medskip

\nin
The following result lies in the heart of the proof of the theorem.

\smallskip

\begin{lemma}[Additivity Lemma] \label{l:add}
Suppose inequalities~$(\circ)$ as in the theorem hold for pairs $(\ell,m_1)$
and~$(\ell,m_2)$.  Suppose also that at least one of integers $\{\ell,m_1,m_2\}$ is even, 
and at least one $\geq 3$.  Then~$(\circ)$ holds for $(\ell,m_1+m_2)$.
\end{lemma}

\smallskip

Although stated combinatorially, the only proof we know is algebraic.
We first establish the lemma and then combine it with computational
results to derive the theorem.  We conclude with historical remarks,
brief overview of the literature and open problems.

\newpage

\section{Kronecker coefficients}

\nin
We adopt the standard notation in combinatorics of partitions
and representation theory of~$S_n$ (see e.g.~\cite{Mac,Sta}).
We use $g(\la,\mu,\nu)$ to denote the \emph{Kronecker coefficients}:
$$
\chi^\la \ts \otimes \ts \chi^\mu \, = \, \sum_{\nu \vdash n}
\, g(\la,\mu,\nu)\. \chi^\nu\., \quad \text{where} \ \ \la,\mu \vdash n\ts.
$$
The following technical result was never stated before,
but is implicit in~\cite{PP} (see also~\cite[$\S 4$]{Val}).

\begin{lemma}\label{l:g_partitions}
Let $n=\ell \ts m$, $\tau_k=(n-k,k)$, where $0\leq k\leq n/2$ and set $p_{-1}(\ell,m)=0$.  Then
$$g(m^\ell,m^\ell,\tau_k) \, = \, p_k(\ell,m) \. - \. p_{k-1}(\ell,m)\ts.
$$
\end{lemma}

\begin{proof}
Let $\la \vdash n$, $\pi\vdash k$ and $\theta\vdash n-k$, and
let ``$*$'' denote the Kronecker product of symmetric functions.
Littlewood's formula %~\cite{Lit}
states that
$$
s_{\lambda}*(s_{\pi}\ts s_{\theta}) \. = \.
\sum_{\alpha\vdash k\ts, \.\beta \vdash n-k } \.
c^{\lambda}_{\alpha\ts\beta}\ts (s_{\alpha}*s_{\pi})\ts (s_{\beta}*s_{\theta})\.,
$$
where $c^{\lambda}_{\mu\ts\nu}$ denote the Littlewood--Richardson
coefficients.  Clearly, \ts $s_{\nu}*s_{a}=s_{\nu}$\ts, for all $\nu\vdash a$.
We obtain:
$$
s_{\lambda}\ts *\ts (s_{k}\ts s_{n-k}) \. = \.
\sum_{\alpha \vdash k, \. \beta \vdash n-k} \. c^{\lambda}_{\alpha\ts \beta}  \.
s_{\alpha}\ts s_{\beta} \. =\.
\sum_{\alpha\vdash k, \. \beta \vdash n-k, \. \nu \vdash n} \.
c^{\lambda}_{\alpha\ts\beta} \. c^{\mu}_{\alpha\ts\beta}\. s_{\mu}\ts.
$$
By the Jacobi--Trudi formula, we have:
$$
s_{\tau_k} \. = \. s_{(n-k,k)} \. = \. s_k \ts s_{n-k} \. - \. s_{k-1} \ts s_{n-k+1}\..
$$
This gives:
$$
s_{\lambda}*s_{\tau_k} \. = \.
s_{\lambda}\ts * \ts (s_k\ts s_{n-k}) \. - \. s_{\lambda}\ts * \ts (s_{k-1}\ts s_{n-k+1}) \. = \.
\sum_{\mu \vdash n} \. a_k(\lambda,\mu)\ts s_{\mu} \. - \. \sum_{\mu \vdash n} \. a_{k-1}(\lambda,\mu)\ts s_{\mu}\ts,
$$
where
$$
a_k(\lambda,\mu) \, = \,
\sum_{\alpha \vdash k, \. \beta\vdash n-k} \, c^{\lambda}_{\alpha\ts\beta}\. c^{\mu}_{\alpha\ts\beta}\..
$$
Taking the coefficient at $s_{\mu}$ in the expansion of \ts $s_{\lambda}*s_{\tau_k}$
in terms of Schur functions, we get:
$$(\divideontimes) \qquad
g(\lambda,\mu,\tau_k) \. = \. a_k(\lambda,\mu) \. - \. a_{k-1}(\lambda,\mu)\ts.
$$
Let $\lambda=\mu=(m^\ell)$.  Recall that
$c^{(m^\ell)}_{\alpha\beta}=1$ if $\al$ and~$\beta$ are complementary
partitions within the  rectangle~$(m^\ell)$; and~$c^{(m^\ell)}_{\alpha\beta}=0$ otherwise
(see e.g.~\cite{MY}).  Therefore,
$$
a_k(m^\ell,m^\ell) \, = \, \sum_{\alpha \vdash k, \, \alpha \subset (m^\ell)}\. 1^2 \, = \. p_k(\ell,m)\ts.
$$
Substituting this into~$(\divideontimes)$, gives the result.  \end{proof}

\smallskip

\begin{thm}[Semigroup property]  \label{t:semi}
Suppose $\la,\mu,\nu,\al,\be,\ga$ are partitions of~$n$, such that \\
$\ts g(\la,\mu,\nu) >0\ts $ and $\ts g(\al,\be,\ga)> 0$. Then
$\ts g(\la+\al,\mu+\be,\nu+\ga)\ts >\ts 0$.
\end{thm}

\smallskip

\begin{rem}\label{r:man} {\rm
This result was conjectured by Klyachko in~2004, and recently proved
in~\cite{CHM}.  It is the analogue of the semigroup property
of Littlewood--Richardson coefficients proved by Brion and Knop
in~1989 (see~\cite{Zel} for the history and the related results).
Unfortunately, the Knutson--Tao \emph{saturation theorem} does
not generalize to Kronecker coefficients (see e.g.~\cite[$\S 2.5$]{K2}).
Let us mention the following useful extension by Manivel~\cite{Man}:
\ts in conditions of the theorem, we have
$$
g(\la+\al,\mu+\be,\nu+\ga)\, \ge \, \max\bigl\{\ts g(\la,\mu,\nu), \ts g(\al,\be,\ga)\ts\bigr\}\ts.
$$
}
\end{rem}
% \vskip.6cm

\bigskip

\section{The proofs}

\begin{proof}[Proof of Lemma~\ref{l:add}]
Let $\la=\mu = (m_1^\ell)$, $\al=\be = (m_2^\ell)$, $\nu=(\ell\ts m_1-r,r)$,
$\ga=(\ell\ts m_2-s,s)$.  By the strict unimodality assumption for $(\ell,m_1)$
and $(\ell,m_2)$ and Lemma \ref{l:g_partitions}, we have
$$
g(m_1^\ell,m_1^\ell,\nu) \. > \. 0, \qquad g(m_2^\ell,m_2^\ell,\ga) \. > \. 0,
$$
for all $r,s \ge 0, \ts \ne 1$. Apply Theorem~\ref{t:semi} to the fixed partitions above.  Now, for all $k=r+s$, we then have
$$
g\bigl((m_1+m_2)^\ell,(m_1+m_2)^\ell,\tau_k\bigr) \. = \.
g(m_1^\ell+m_2^\ell,m_1^\ell+m_2^\ell,\nu+\ga) \. > \. 0,
$$
where $n=(m_1+m_2)\ell$ and $\tau_k=(n-k,k)$ as before.
For $k\leq 3$ we can choose $(r,s)=(0,k)$ or $(k,0)$, as at it is implicit that $\ell,m_1,m_2\geq 2$ and one of them is $\geq 3$.
For $3<k\leq \lfloor n/2 \rfloor-1$ we have that $k\leq \lfloor \ell m_1/2 \rfloor + \lfloor \ell m_2/2\rfloor$, so there are values $r,s\geq 2$, $r \leq  \lfloor \ell m_1/2 \rfloor $ and $s \leq  \lfloor \ell m_2/2 \rfloor$, such that $k=r+s$. Finally, when $k=  \lfloor n/2 \rfloor$, by the parity conditions we have that at least one of $\ell m_1$, $\ell m_2$ is even, so we can choose $(r,s)=(\ell m_1/2, \lfloor \ell m_2/2 \rfloor)$ or $(\lfloor \ell m_1/2 \rfloor, \ell m_2/2)$.
%Since such~$r$ and~$s$ exist for every $k\neq 1$, and $m_1,m_2>2$, we have the result.
\end{proof}

\begin{thm}\label{t:rest}
Let $m, \ell \ge 2$.
Strict unimodality~$(\circ)$ as in Theorem~\ref{t:strict} holds for pairs $(\ell,m)$, $\ell\le m$, 
if and only if \ts $\ell=m=2$, or \ts $\ell,m\ge 5$ \ts with the exception of the following values:
$$
\{\ts(5,6)\ts, \. (5,10)\ts, \. (5,14)\ts, \. (6,6)\ts, \. (6,7)\ts, \. (6,9)\ts, \. (6,11)\ts, \. (6,13)\ts, \. (7,10)\ts\}\ts.
$$
\end{thm}

% This theorem covers few other cases not covered by Theorem~\ref{t:strict}.

\begin{proof}  A direct calculation gives strict unimodality for each $\ell\in \{8,\ldots,15\}$, and $8\le m < 16$.
For each fixed $\ell \in \{8,\ldots,15\}$ and $m\ge 16$, we have that $m=8a +b$ for $a\ge 1$ and $8 \le b <16$. Applying the additivity lemma successively  with $\ell, m_1=8k+b, m_2=8$ for $k=0,1,\ldots,a-1$, shows that $(\circ)$ holds for all $\ell \in  \{8,\ldots,15\}$ and $m\ge 16$.

Fixing any $m\geq 8$ and applying the additivity lemma in the direction of $\ell$ the same way by expressing $\ell = 8a'+b'$, shows that $(\circ)$ holds for all $m,\ell \geq 8$.

A direct calculation also gives strict unimodality for all values of $\ell \in \{5,6,7\}$ and $5\leq m \leq 20$ with the exception of the listed cases, where the middle three coefficients  of the expansion of $\binom{\ell+m}{m}_q$ are equal. Now we apply the additivity lemma for each value of $\ell=5,6,7$ and $m= 10a+b$ where  $10\le b \le 19$ and induct over $a$ with the values $m_1=10(a-1)+b$ and $m_2=10$. The cases $\ell>m$ follow from the symmetry.

Now, case $\ell=2$ is straightforward, since $p_{2i}(2,m)=p_{2i+1}(2,m)$
for all~$i<n/4$.  On the other hand, cases $\ell=3,4$ have been studied
in~\cite{Lin,West} using an explicit symmetric chain decomposition.  Since all
chain lengths there are~$\ge 3$, we obtain equalities for the middle coefficients.
%The remaining cases $m < \ell$ follow by the symmetry of $q$-binomial coefficients.
\end{proof}

\smallskip

% \newpage

\section{Final remarks} \label{s:fin}

\subsection{}  Let us quote a passage from~\cite{Syl} describing
how Sylvester viewed his work:

\smallskip

\begin{quote} ``\ts I am about to demonstrate a theorem which has been waiting proof for
the last quarter of a century and upwards.~[\ts ...\ts ]\ts~I accomplished with scarcely an
effort a task which I had believed lay outside the range of human power.''
\end{quote}

\smallskip

\noindent
The grandeur notwithstanding, it does reveal Sylvester's
excitement over his discovery.

\subsection{}
Proving unimodality is often difficult and involves a remarkable diversity
of applicable tools, ranging from analytic to bijective, from topological
to algebraic, and from Lie theory to probability.
We refer to~\cite{B1,B2,Sta-unim} for a broad overview of the subject.

\subsection{} \label{ss:fin-basis}
The Additivity Lemma gives an example of a $2$-dim \emph{Klarner system}, which
always have a finite basis (see~\cite{Reid}).

\subsection{}
The equation (KOH) in~\cite{Zei}, based on O'Hara's
combinatorial approach to unimodality of $q$-binomial coefficients~\cite{O},
gives a useful recurrence relation (cf.~\cite{K1,M1}).
It would be interesting to see if (KOH) can be used to prove
Theorem~\ref{t:strict}.

\subsection{}
An important generalization of Sylvester's theorem is the unimodality of
 $s_\la(1,q,\ldots,q^{m})$ as a polynomial in $q$, see~\cite[p.~137]{Mac}.
We conjecture that if the \emph{Durfee
square} size of~$\la$ is large enough, then these coefficients
are strictly unimodal. An analogue of~(KOH) in this case is in~\cite{K1}.

\subsection{}
In a different direction, we believe that for every $d\ge 1$ there exists $L(d)$,
s.t.~$p_{k}(\ell,m) - p_{k-1}(\ell,m)\ge d$ for all $L(d) < k \le \ell\ts m/2$,
and $m,\ell$ large enough.  Unfortunately, the tools in this paper are not directly
applicable.  However, for $\ell=m$, this follows from Prop.~11 in~\cite{Sta-unim},
and further extension of Thm.~5.2 in~\cite{PP}
on strict unimodality of the number of partitions into distinct odd parts.
Then, combined with Manivel's extension (see Remark~\ref{r:man}),  and the finite
basis theorem (see $\S$\ref{ss:fin-basis}), this would prove the conjecture
in a similar manner as the proof of Theorem~\ref{t:rest}.  We plan to return to this
problem in the future.

% \newpage

\vskip.44cm

{\small
\noindent
{\bf Acknowledgements.}
We are grateful to Stephen DeSalvo, Richard Stanley and Ernesto Vallejo 
for interesting conversations and helpful remarks.  We are especially 
thankful to Fabrizio Zanello for pointing out the error in the original
statement of Theorem~\ref{t:rest}.  
The first author was partially supported by  BSF and NSF grants,
the second by a Simons Postdoctoral Fellowship.
}

% \newpage

\vskip.8cm

%%%%%%%%%%%%%%%%%%%%%%%%%%%%%%%%%%%%%%%%%%%%%%%%%%%%%%%%%%%%%%%%%%%%%%%%

%%%%%%%%%%%%%%%%%%%%%%%%%%%%%%%%%%%%%%%%%%%%%%%%%%%%%%%%%%%%%%%%%%%%%%%%

{\footnotesize

\begin{thebibliography}{13}\label{refpage}

\bibitem{B1}
F.~Brenti, Unimodal, log-concave, and P\'{o}lya frequency sequences in combinatorics,
\emph{Mem.~AMS}, No.~413, 1989, 106~pp.

\bibitem{B2}
F.~Brenti,
Log-concave and unimodal sequences in algebra, combinatorics, and geometry: an update,
in \emph{Contemp. Math.}~\textbf{178}, AMS, Providence, RI, 1994, 71--89.

\bibitem{CHM}
M.~Christandl, A.~W.~Harrow, G.~Mitchison,
Nonzero Kronecker coefficients and what they tell us about spectra,
\emph{Comm. Math. Phys.}~\textbf{270} (2007), 575--585.

\bibitem{K1}
A.~N.~Kirillov, Unimodality of generalized Gaussian coefficients,
\emph{C.R.~Acad. Sci. Paris S\'{e}r.~I~Math.}~\textbf{315}~(1992),
no.~5, 497--501.

\bibitem{K2}
A.~N.~Kirillov,
An invitation to the generalized saturation conjecture,
\emph{Publ. RIMS}~\textbf{40} (2004), 1147--1239.

\bibitem{Lin}
B.~Lindstr\"{o}m,
A partition of $L(3,n)$ into saturated symmetric chains,
\emph{Eur.~J.~Combin.}~\textbf{1} (1980), 61--63.

\bibitem{M1}
I.~G.~Macdonald,
An Elementary Proof of a $q$-Binomial Identity, in
\emph{$q$-Series and Partitions} (IMA, Vol.~18),
Springer, New York, 1989, 73--75.

\bibitem{Mac}
I.~G.~Macdonald, \emph{Symmetric functions and Hall polynomials}
(Second ed.), Oxford U.~Press, New York, 1995.

\bibitem{Man}
L.~Manivel, On rectangular Kronecker coefficients,
\emph{J.~Algebraic Combin.}~\textbf{33} (2011), 153--162.

\bibitem{MY}
H.~Mizukawa, H.-F.~Yamada,
Rectangular Schur functions and the basic representation of affine Lie algebras,
\emph{Discrete Math.}~\textbf{298} (2005), 285--300.

\bibitem{O}
K.~M.~O'Hara, Unimodality of Gaussian coefficients: a constructive proof,
\emph{J.~Combin.~Theory, Ser.~A}~\textbf{53} (1990), 29--52.

\bibitem{PP}
I.~Pak, G.~Panova,
Unimodality via Kronecker products, \ts
{\tt arXiv:}{\tt 1304.5044}.

\bibitem{PPV}
I.~Pak, G.~Panova, E.~Vallejo,
Kronecker products, characters, partitions, and the tensor square conjectures,
\ts {\tt arXiv:}{\tt 1304.0738}.

\bibitem{Reid}
M.~Reid, Klarner Systems and Tiling Boxes with Polyominoes,
\emph{J.~Combin. Theory, Ser.~A}~\textbf{111} (2005), 89--105.

\bibitem{Sta-unim}
R.~P.~Stanley,
Log-concave and unimodal sequences in algebra, combinatorics, and geometry,
in \emph{Ann. New York Acad. Sci.}~\textbf{576}, New York Acad. Sci., New York, 1989,
500--535.

\bibitem{Sta}
R.~P.~Stanley, \emph{Enumerative Combinatorics}, Vol.~2,
Cambridge U.~Press, Cambridge, 1999.

\bibitem{Syl}
J.~J.~Sylvester,
Proof of the hitherto undemonstrated Fundamental Theorem of Invariants,
\emph{Philosophical Magazine}~\textbf{5} (1878), 178--188;
reprinted in \emph{Coll. Math. Papers}, vol.~3, Chelsea, New York, 1973, 117--126;
available at \ts {\tt http://tinyurl.com/c94pphj}


\bibitem{Val}
E.~Vallejo, A diagramatic approach to Kronecker squares, 
{\tt arXiv:1310.8362}.

\bibitem{West}
D.~B.~West,
A symmetric chain decomposition of $L(4,n)$,
\emph{Eur.~J.~Combin.}~\textbf{1} (1980), 379--383.

\bibitem{Zei}
D.~Zeilberger, Kathy O'Hara's constructive proof of the unimodality of the
Gaussian polynomials, \emph{Amer. Math. Monthly}~\textbf{96} (1989), 590–-602.

\bibitem{Zel}
A.~Zelevinsky,
Littlewood-Richardson semigroups, in
\emph{New Perspectives in Algebraic Combinatorics},
Cambridge U.~Press, Cambridge, 1999, 337--345.

\end{thebibliography}

}
\end{document}